%% file: TRS_Solving_0711.tex
\documentclass[review,onefignum,onetabnum]{siamart190516}
\usepackage{amsmath}
\usepackage{amssymb}
\usepackage{mathtools}

\usepackage{graphicx}
\usepackage{mathrsfs, lscape, color, float, graphicx, amsfonts, mathrsfs,amscd, dsfont, bm, subfig, multirow}

\usepackage{listings}
\usepackage{enumerate}
\usepackage{color}

\def\ST{{\rm s.t.}}

\def\R{{\mathbb R}}

\newtheorem{observ}{Observation}[section]
\newtheorem{rem}{Remark}[section]

\def\ST{\songti\rm\relax}
\def\R{{\mathbb R}}

\def\ST{{\rm s.t.}}

\input{ex_shared}

\ifpdf
\hypersetup{
  pdftitle={On globally solving nonconvex trust region subproblem via projected gradient method},
  pdfauthor={}
}
\fi

\begin{document}

\maketitle

% REQUIRED
\begin{abstract}
The trust region subproblem (TRS) is to minimize a possibly nonconvex quadratic function over a Euclidean ball. There are typically two cases for (TRS), the so-called  ``easy case'' and ``hard case''. Even in the ``easy case'', the sequence generated by the classical projected gradient method (PG) may converge to a saddle point at a sublinear local rate, when the initial point is arbitrarily selected from a nonzero measure feasible set. To our surprise, when applying (PG) to solve a cheap and possibly nonconvex reformulation of (TRS), the generated sequence initialized with {\it any} feasible point almost always converges to its global minimizer. The local convergence rate is at least linear for the ``easy case'', without assuming that we have possessed the information that the ``easy case''  holds.  We also consider how to use (PG) to globally solve equality-constrained (TRS).
\end{abstract}
% REQUIRED

% REQUIRED

\section{Introduction}
\label{intro}
The trust region subproblem (TRS) plays a great role in the trust
region method for solving nonlinear programming problems, see \cite{Conn00,yuan15}.
Typically, we can write it as the following (possibly nonconvex) quadratic optimization over the unit Euclidean ball:
\begin{equation}\label{eq:TRS}\tag{TRS}
	\begin{array}{cl}
		\min \left\{q(x)=\frac{1}{2}x^THx+c^Tx:~ x\in B_n=\{x\in\R^n: x^Tx\le 1\}\right\},
	\end{array}
\end{equation}
where $H=H^T\in \R^{n\times n}$ and $c\in \R^n$. When $H$ is not positive semidefinite, \eqref{eq:TRS} could have a (unique) local non-global minimizer, see \cite{Mart94,Wang9020} for full characterizations. However, nonconvex \eqref{eq:TRS} could be efficiently and globally solved based on the necessary and sufficient global optimality condition \cite{Gay81,More83,Sorensen82} or the hidden convexity property, see \cite{Xia2020} and references therein. In the literature, there are numerous  algorithms for globally solving \eqref{eq:TRS}, including the approaches based on finding the zero point of the secular function in terms of the dual variable by Newton's method \cite{More83}, by bisection method \cite{Hazan16} or by a hybrid algorithm combining the former two methods \cite{Y92}, the Lanczos methods  \cite{Carmon18,Gould99,Zhang17}, the sequential subspace method \cite{Hager01}, the eigenvector based methods  \cite{Adachi17,Sorensen97}, and so on. A notable simple first-order method is first computing the minimum eigenvalue of $H$, denoted by $\lambda_{1}$ throughout this paper, and then employing  Nesterov's accelerated gradient algorithm to solve the convex programming reformulation of \eqref{eq:TRS}:
\begin{equation}\label{eq:C}\tag{C}
	\begin{array}{cl}
		\min \left\{q(x)+\min\{\frac{\lambda_{1}}{2},0\}(1- x^Tx):\, x\in B_n\right\}.
	\end{array}
\end{equation}
This two-stage approach is independently proposed in \cite{Ho-Nguyen17,wang17}, while the convex reformulation  \eqref{eq:C} dates back to \cite{Flippo96,Y92}.

The classical projected gradient algorithm (PG) is an efficient first-order method to solve convex \eqref{eq:TRS}. It reads as
\begin{equation}\label{eq:pro}\tag{PG}
	x^{k+1}=P_{B_n}(x^k-\eta\nabla q(x^k))=\left\{
	\begin{aligned}
		&x^k-\eta\nabla q(x^k),&&\ {\rm if}\   \|x^k-\eta\nabla q(x^k)\|\le1\\
		&\frac{x^k-\eta\nabla q(x^k)}{\|x^k-\eta\nabla q(x^k)\|},&&\ {\rm otherwise},
	\end{aligned}
	\right.
\end{equation}
for $k=0,1,2, \cdots$,
where $x^0\in B_n$ is the initial point, $\eta\in(0, 2/L)$ is the step size, $L={\|H\|}_2$ (the spectral norm of $H$) is the Lipschitz constant so that the gradient $\nabla q(x)$ is Lipschitz continuous with the constant $L$, and $P_Q{(\cdot)}$ is the operator of metric projection onto the compact set $Q$ with respect to the Euclidean norm $\|\cdot\|$, i.e., for $y\in\R^n$,
\[
P_Q{(y)}:=\left\{x\in Q:~ \|x-y\|={\rm min}_{z\in Q} \|z- y\|\right\}.
\]
In our case, the feasible set $B_n$ is convex so that $P_{B_n}{(y)}$ is a single-point set and has a closed-form expression for all $y\in\R^n$.
The efficiency can be observed from the fact that the computation in \eqref{eq:pro} only relies on matrix-vector products.

One attempt of extending  \eqref{eq:pro} to solve nonconvex \eqref{eq:TRS} is the D.C. (difference-of-convex) scheme proposed in \cite{Tao98}. However, there is no guarantee that the returned solution is a global minimizer. Recently,
Beck and Vaisbourd \cite{Beck2018} studied a class of first-order methods, including \eqref{eq:pro} as an important case, for globally solving \eqref{eq:TRS}. They proved that in the ``easy case'', the sequence $\{x^k\}_{k=0}^{\infty}$ generated by \eqref{eq:pro} with $x^0=0$ converges to the globally optimal solution, and in the ``hard case'', the sequence $\{x^k\}_{k=0}^{\infty}$ generated by \eqref{eq:pro} converges to the globally minimizer with probability one when $x^0$ is uniformly and randomly selected from $B_n$. Based on these observations, Beck and Vaisbourd \cite{Beck2018} proposed double-start \eqref{eq:pro} for globally solving \eqref{eq:TRS}. More precisely, although it is unknown which case (``easy case'' or ``hard case'') holds, it is sufficient to employ \eqref{eq:pro} twice, each with a different initialization, and then return the better of the two obtained points. Compared with the two-stage approach based on  \eqref{eq:C}, double-start \eqref{eq:pro} has the benefit of running in parallel.

At the beginning of this study, we show by examples that two disadvantages exist for double-start \eqref{eq:pro}. Firstly, the worst convergence rate of double-start \eqref{eq:pro} even for the ``easy case'' is only sublinear. It is known that in the ``easy case'', the local convergence rate of $\{x^k\}_{k=0}^{\infty}$ generated by \eqref{eq:pro} with $x^0=0$ is linear \cite{Jiang22} as the sequence converges to the global minimizer.  However,
we can use an example in the ``easy case'' to show that
in the second run of  double-start \eqref{eq:pro} with $x^0$ being uniformly and randomly selected from $B_n$, $\{x^k\}_{k=0}^{\infty}$  could converge locally sublinearly to a saddle point with a probability greater than zero.
Secondly, the final two objective function values returned by
double-start \eqref{eq:pro} require a high degree of accuracy
for correctly comparing.  Otherwise, one may mistake a saddle point or a local non-global minimizer for the global minimizer. We have an example to show that the gap between local non-global minimum and global minimum can be arbitrarily small.

The main contribution of this research is to  present a
%Comparing with two-stage approach and double-start \eqref{eq:pro}, our method can be regarded as
{\it one-stage} and {\it single-start} \eqref{eq:pro}.
We first present a cheap but novel reformulation of \eqref{eq:TRS}, which is a ($2n$)-dimensional \eqref{eq:TRS}. Though it is possibly nonconvex, the new reformulation has the nice property that any second-order stationary point (including local minimizer) is globally optimal.
We prove that the sequence obtained by applying \eqref{eq:pro} to solve the new reformulation almost always converges to its global minimizer for  both ``easy case'' and ``hard case''.
If the ``easy case'' holds (though we do not possess this information), the local convergence rate is at least linear. %It should be noted that we do not possess the information  on which case holds.
Finally, the global minimizer of original \eqref{eq:TRS} is  easily recovered by a closed-form expression.

As an extension, we consider globally solving the equality-constrained \eqref{eq:TRS}:
\begin{equation}\label{eq:TRS_e}\tag{TRSe}
	\min \left\{q(x)=\frac{1}{2}x^THx+c^Tx:\, x\in \partial B_n=\{x\in\R^n: x^Tx=1\}\right\},
\end{equation}
which itself  has fruitful applications \cite{PY2020}. Different from \eqref{eq:TRS}, the feasible region of \eqref{eq:TRS_e} is nonconvex.
Generalized projected gradient method (GPG) has been presented for  solving the optimization problems over  the nonconvex set, see \cite{Polyak19,Jain17}. For solving \eqref{eq:TRS_e}, the iterative scheme of (GPG)  is given by
\begin{equation}\label{eq:proe}\tag{PGe}
	x^{k+1}=P_{\partial B_n}(x^k-\eta_e\nabla q(x^k))=
	\frac{x^k-\eta_e\nabla q(x^k)}{\|x^k-\eta_e\nabla q(x^k)\|},
\end{equation}
where $\eta_e\in(0, 1/L]$ is the step size. If it holds that $x^k-\eta_e\nabla q(x^k)=0$, then the iteration stops as $x^k$ is already a stationary point. With additional assumptions made in \cite[Theorems 1 and 2]{Polyak19},  (GPG) is guaranteed to find the global minimizer. It is observed that the additional assumptions required in \cite[Theorems 1 and 2]{Polyak19} are too restrictive for \eqref{eq:proe}.
Jain and Kar claimed in  \cite[Theorem 3.3]{Jain17} that the sequence generated by (GPG) converges to the global minimizer under their assumptions. It is, however, not correct as illustrated by an example of \eqref{eq:TRS_e}.
On the other hand,  \eqref{eq:TRS_e} is an  optimization problem on the manifold $\partial B_n$ and can be solved by Riemannian gradient method (RG), see \cite{Absil08}. Lee et al. \cite{Lee2019} proved that  (RG) almost always avoids strict saddle points, where the step size has been corrected in \cite{Zheng22}.
However, \eqref{eq:TRS_e} could have a non-strict saddle point or even a local non-global minimizer which can not be avoided, see examples in \cite{Wang9020}. So  (RG) may fail to find the global minimizer.
In the second part of this research, we first show that \eqref{eq:TRS_e} can be globally solved by employing \eqref{eq:proe} to solve a reformulation  similar to that of \eqref{eq:TRS}. Finally, to our surprise, we can build a cheap \eqref{eq:TRS}-reformulation of \eqref{eq:TRS_e} so that it can be  globally solved with a step size larger than that of \eqref{eq:proe}.

In the following, we list the contributions of this study.
\begin{itemize}
\item Two disadvantages of double-start \eqref{eq:pro} are illustrated by examples.
\item We cheaply lift \eqref{eq:TRS} to an equivalent (possibly nonconvex) \eqref{eq:TRS} in $\R^{2n}$ so that any second-order stationary point (including local minimizer) is globally optimal.
The global minimizer of the original \eqref{eq:TRS} can be recovered from that of the new reformulation through a closed-form expression.
	\item When solving the new reformulation of \eqref{eq:TRS} by  \eqref{eq:pro} with an initial point uniformly and randomly selected from $B_{2n}$, we prove that the generated sequence converges to the global minimizer with probability one.
	\item  In the ``easy case'',  the local convergence rate of our approach is linear, while it could be sublinear for the second run of double-start \eqref{eq:pro}.
	\item We generalize \eqref{eq:pro} to globally solve the  new similar reformulation of \eqref{eq:TRS_e}.
	\item
As a new approach for solving  \eqref{eq:TRS_e}, we apply \eqref{eq:pro} to globally solve a novel cheap \eqref{eq:TRS}-reformulation of \eqref{eq:TRS_e},  with a step size larger than that of the generalized \eqref{eq:pro}.
\end{itemize}

The remainder of this paper is organized as follows. Section 2 first presents classical optimality conditions for \eqref{eq:TRS}, and then gives two examples to illustrate two disadvantages of double-start \eqref{eq:pro}. Section 3 establishes an equivalent reformulation of \eqref{eq:TRS} with a nice property and then proves that the iterative sequence generated by \eqref{eq:pro} for solving this new reformulation almost always converges to its global minimizer. Section 4 considers globally solving \eqref{eq:TRS_e}. We conclude the paper in Section 5.

{\bf Notation.} Denote by $v(\cdot)$ the optimal value of the problem $(\cdot)$.
Let $I$ be the identity matrix of proper dimension.
%For any matrix $A\in\mathcal{S}^n$ and any vector $v\in\R^n$, ${\|A\|}_2$ and $\|v\|$ denote the spectral norm of $A$ and the $2-$norm of $v$, respectively.
%For any twice continuously differentiable function $g: \R^n\rightarrow\R$ and $x\in\R^n$, let $\nabla g(x)$ be the gradient of $g$ at $x$.
For a square matrix $B$, $B\succeq 0$ means that $B$ is positive semidefinite,
${B}^\dagger$ stands for the Moore-Penrose pseudoinverse of $B$,   ${\rm tr}(B)$ gives the trace of $B$,  $\lambda_{\min}(B)$ denotes the minimum eigenvalue $B$, and
${\rm Range}(B)$ returns the range (or column) space of $B$. For a vector $x$, $\|x\|$ denotes the Euclidean norm of $x$. For two scalars $a<b$,  $(a,b):=\{x\in\R:~a<x<b\}$ and $(a,b]:=\{x\in\R:~a<x\le b\}$.

\section{Preliminaries}
\subsection{Known results of \eqref{eq:TRS}}
Note that the linear independence constraint qualification (LICQ) always holds at any feasible point of \eqref{eq:TRS}. For $x^*\in\R^n$, if there exists $\lambda^*$ such that KKT condition
\begin{eqnarray}
	&& \lambda^*\ge 0,\ (H+\lambda^*I)x^*+c=0,\label{eq:stationary}\\
	&& \lambda^*(x^{*T}x^*-1)=0,\  x^{*T}x^*-1\le0\label{eq:com}
\end{eqnarray}
hold, then $x^*$ is called a stationary point, the pair $(x^*, \lambda^*/2)$ is called a KKT point, and the nonnegative number $\lambda^*/2$ is called a KKT multiplier corresponding to $x^*$.
According to the classical optimization theory, any local minimizer of \eqref{eq:TRS} must be a stationary point.  A stationary point is called a saddle point if it is not locally optimal.
%stationarity is a necessary (but not sufficient) condition for local optimality.
In 1980s, \eqref{eq:TRS} is proved to enjoy the following necessary and sufficient condition at its global minimizer.
\begin{lemma}[\cite{Gay81,More83,Sorensen82}]\label{le:global}
$x^*$ is a global minimizer of \eqref{eq:TRS} if and only if there exists a unique $\lambda^*$ such that
$(x^*, \lambda^*/2)$ is a KKT point and $\lambda^*\ge-\lambda_1$, where $\lambda_{1}$ is the minimum eigenvalue of $H$.
\end{lemma}
%\begin{lemma}[\cite{Mart94,Wang9022}]\label{le:multi}
%The KKT multipliers corresponding to all global minimizers of \eqref{eq:TRS} are equal.
%\end{lemma}

In the literature,  if $c\notin{\rm Range}(H-\lambda_1 I)$, it is called that ``easy case'' holds for \eqref{eq:TRS}, and ``hard case'' otherwise.
Let $x^*$ be a global minimizer of \eqref{eq:TRS}, and $\lambda^*/2$ be the corresponding KKT multiplier. If the ``easy case''  occurs, by Lemma \ref{le:global} and $c\notin{\rm Range}(H-\lambda_1 I)$, it must hold that $\lambda^*>-\lambda_1$.
The ``hard case'' can be further split into three subcases, see details in Table \ref{tab:tab1}, where only the third case is called ``ill case''.  Correspondingly, ``easy case'' and the first two subcases of ``hard case'' are collectively referred to as ``well case''.
\begin{table}[h!]
	\begin{center}
		\caption{Three subcases of ``hard case'' for \eqref{eq:TRS}.}\label{tab:tab1}
		\begin{tabular}{|c|c|c|}
			\hline
			\multicolumn{3}{|c|}{ hard case $(c\in{\rm Range}(H-\lambda_1 I))$} \\
			\hline
			hard case (i) & hard case (ii)  &hard case (iii) (\textbf{ill case})\\
			\hline
			$\lambda^*>-\lambda_1 $&  $\lambda^*=-\lambda_1, ~\|(H-\lambda_1 I)^\dagger c\|<1$ & $\lambda^*=-\lambda_1, ~\|(H-\lambda_1 I)^\dagger c\|=1$\\
			\hline
		\end{tabular}
	\end{center}
\end{table}

If  the iterative sequence generated by \eqref{eq:pro} converges to a global minimizer of \eqref{eq:TRS} $x^*$, Jiang and Li \cite{Jiang22} proved that the local convergence rate is at least linear and sublinear for  ``well case''  and ``ill case'', respectively.
\begin{lemma}{\rm (\cite[Theorem 5.1]{Jiang22})}\label{le:rate}
	Let $\{x^k\}_{k=0}^{\infty}$ be the sequence generated by \eqref{eq:pro} with the constant step size $\eta\in (0, 2/L)$ and assume that $\{x^k\}_{k=0}^{\infty}$ converges to $x^*$.  Then for any given $\epsilon>0$, there exists a sufficiently large  positive integer $K$ such that $\{x^k\}_{k\ge K} \subset B(x^*, \epsilon)=\{x:~\|x-x^*\|\le\epsilon\}$, and in the ``ill case'',	it holds that \[f(x^k)-f^*\le\frac{1}{{\left(\frac{k-K}{2M^2+\frac{3}{2}\sqrt{f(x^K)-f^*}}+\frac{1}{\sqrt{f(x^K)-f^*}}\right)}^2}, ~\forall k\ge K.\]
	Otherwise, we have
	\[
	f(x^k)-f^*\le{\left(\frac{M^2}{M^2+1}\right)}^k(f(x^K)-f^*), ~\forall k\ge K,
	\]
	where $M$ is a constant related to the input data in \eqref{eq:TRS} and the step size $\eta$.
\end{lemma}
\subsection{Disadvantages of double-start \eqref{eq:pro}}\label{sec:beck}
Double-start \eqref{eq:pro} \cite{Beck2018} employs  \eqref{eq:pro} twice with $x^0$ being $0$ and uniformly distributed over $B_n$. The two independently generated iterative  sequences converge to two stationary points, denoted by $x_0^*$ and $x_B^*$, respectively. It is proved in \cite{Beck2018} that  $x_0^*$ is globally optimal in the ``easy case'', and  $x_B^*$ is globally optimal with probability one in the ``hard case''.  Though it is unknown which case holds, one can output the smaller one of $q(x_0^*)$ and $q(x_B^*)$ as the global minimum, as done in double-start \eqref{eq:pro}.

%It should be noted that in the ``easy case'' it is unknown whether $x_B^*$ is globally optimal.
According to Lemma \ref{le:rate},
the local convergence rate of the sequence converging to $x_0^*$ in the ``easy case'' is linear.  However, the local convergence rate of the sequence converging to $x_B^*$ remains unknown in the ``easy case'',  in case that $x_B^*$ is not a global minimizer.

%generated by
%\eqref{eq:pro} under the assumption that it converges to a global minimizer is linear and sublinear for the ``well cases'' and ``ill case'', respectively. When the sequence generated by \eqref{eq:pro} converges to a stationary point of \eqref{eq:TRS}, which is no longer a global minimizer, the local convergence rate is unknown.
Based on an example motivated by \cite{Wang9022},  we have the following observation on the convergence and local convergence rate of double-start \eqref{eq:pro}.
\begin{observ}
%When applying double-start \eqref{eq:pro} to solve \eqref{eq:TRS} in the ``easy case'',  the local convergence rate could be sublinear with a positive probability.
In the ``easy case'' of \eqref{eq:TRS}, the iterative sequence generated by \eqref{eq:pro}
with $x^0$ being uniformly and randomly selected from $B_n$  could converge to a saddle point  with a probability larger than zero. Moreover, the local convergence rate could be sublinear.
\end{observ}

\begin{example}\label{ex:1}
Consider the instance of \eqref{eq:TRS} with $n=2$ and
\begin{equation}\label{eq:ex1}
q(x)=-\frac{13}{2}x_1^2+\frac{13}{2}x_2^2-\frac{250}{169}x_1+\frac{3456}{169}x_2.
\end{equation}
We can verify that the ``easy case'' holds.
With Lemma \ref{le:global},  it can be proved that $x_0^*=(0.687, -0.726)$ is a global minimizer and $x_B^*=(-5/13, -12/13)$  is a saddle point.
%$q(x_0^*)=-15.512$  $q(x_B^*)=-13.732$
\end{example}

We employ \eqref{eq:pro} with the step size $\eta=1/13$ to solve  Example \ref{ex:1}.
As illustrated in Figure \ref{fi:dis},
initialized with $200$ points uniformly and randomly selected from $B_2$, all the independently generated sequences converge to either the global minimizer $x_0^*$ or the saddle point $x_B^*$. It is clearly observed that the initial points  for returning $x_B^*$ build a nonzero measure set in $B_2$.
\begin{figure}
	\centering
	\includegraphics[width=0.7\textwidth]{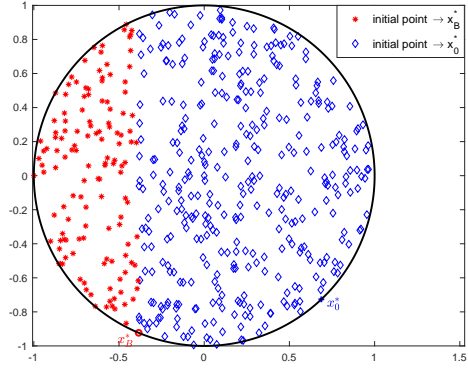}
	\caption{Distribution of $200$ initial points starting from which the iterative sequences generated by \eqref{eq:pro} converge to either the global minimizer $x_0^*$ or the saddle point $x_B^*$.}\label{fi:dis}
\end{figure}
We plot in Figures \ref{fi:rate1} and \ref{fi:rate2} the convergence rates.
One can observe that the rates of convergence to the global minimizer $x_0^*$ and the saddle point $x_B^*$ are linear and sublinear, respectively.
	\begin{figure}[t] \label{fi:cov_rate}
		\centering
		\subfloat[Convergence to the global minimizer $x_0^*$.]{\label{fi:rate1}	\includegraphics[width=0.48\textwidth]{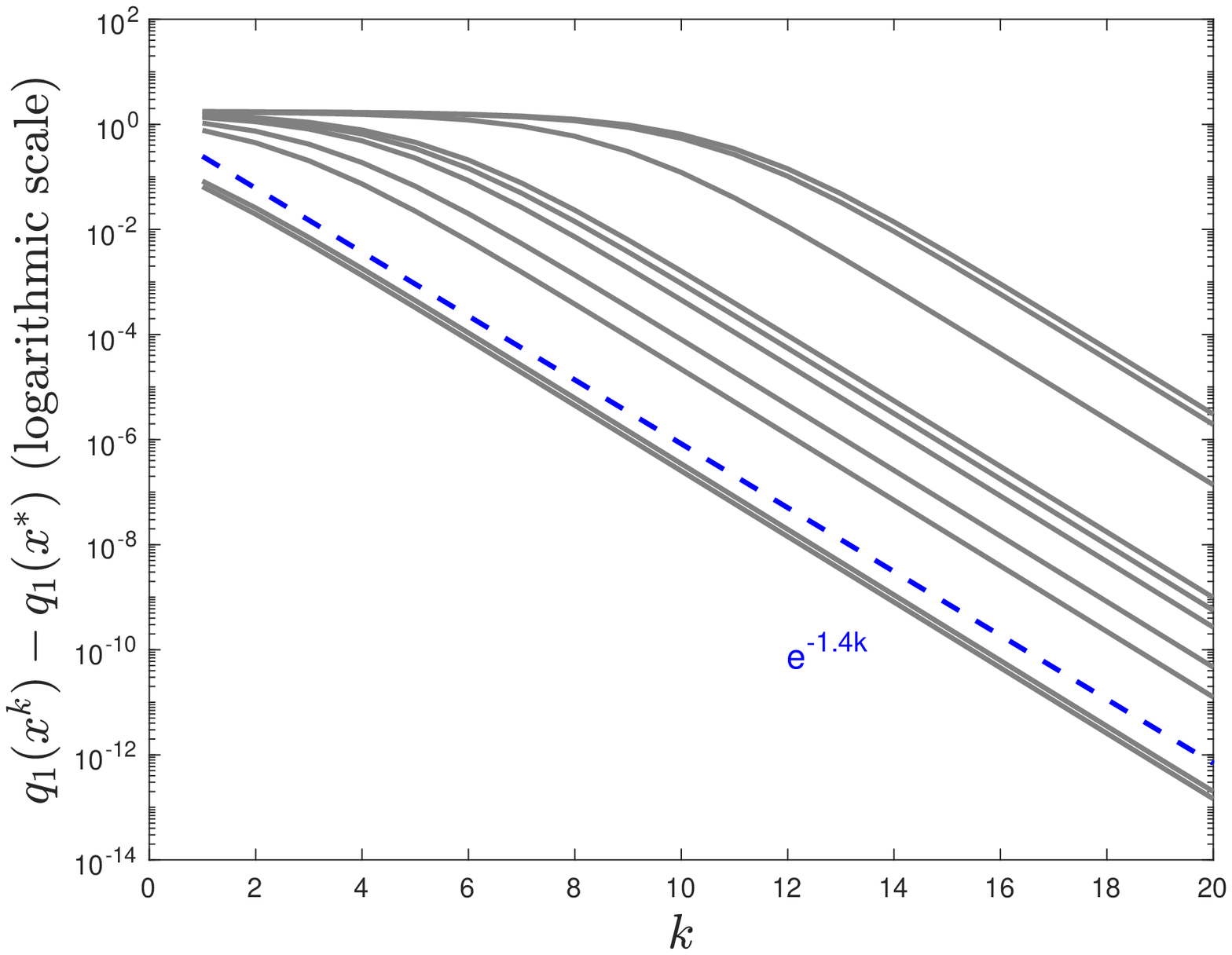}}
		~
		\subfloat[Convergence to the saddle point $x_B^*$.]{\label{fi:rate2}	\includegraphics[width=0.48\textwidth]{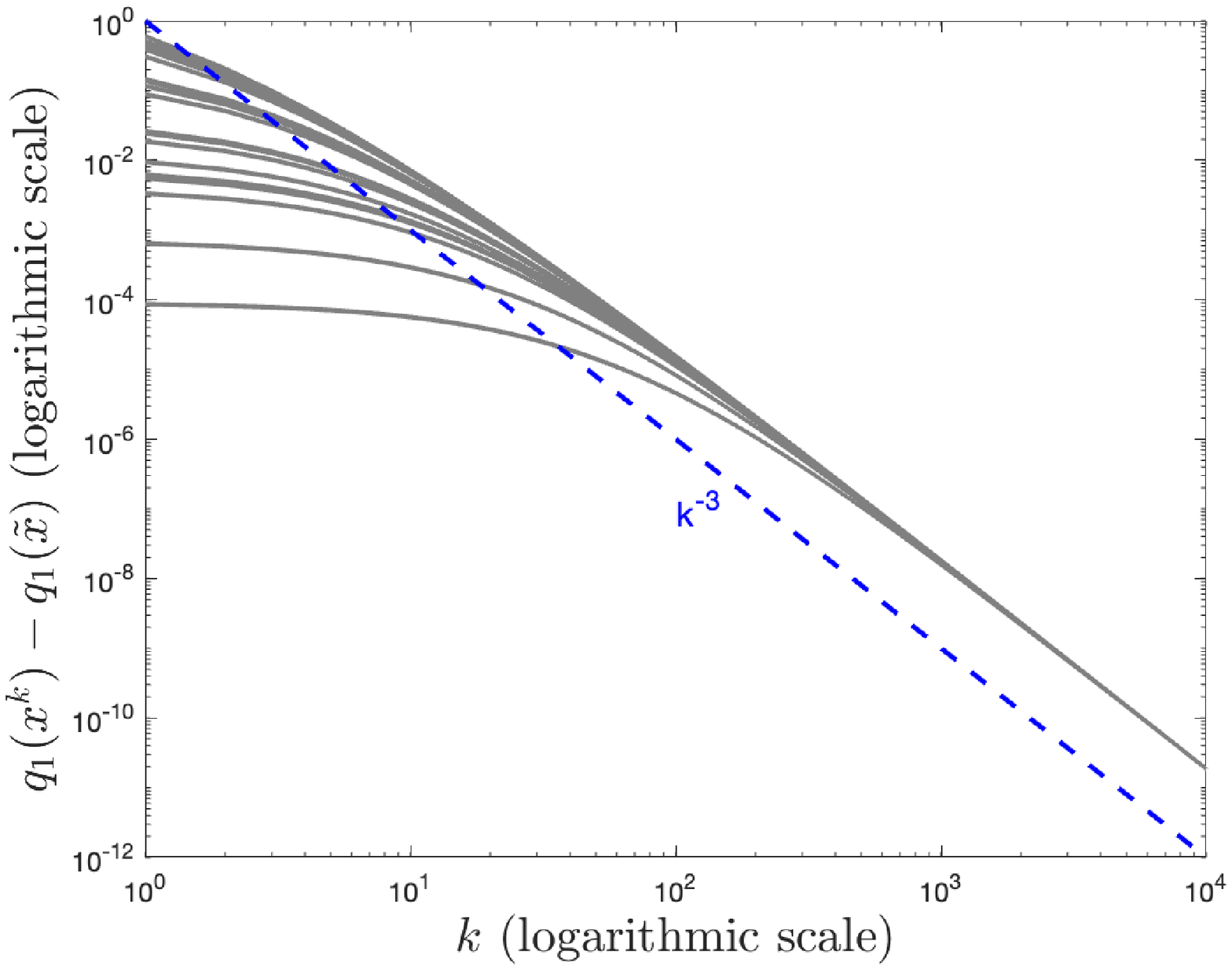}}
		\caption{The rates of convergence to the global minimizer $x_0^*$ and the saddle point $x_B^*$ initialized with different feasible points.}
	\end{figure}

The second disadvantage of double-start \eqref{eq:pro} is that
approximating $q(x_0^*)$ and $q(x_B^*)$ requires high enough  precision so that one can correctly compare both values. Otherwise, double-start \eqref{eq:pro} could mistake saddle point/local non-global minimizer for the global minimizer. The following example implies that the gap between the global minimum and the local non-global minimum could be arbitrarily small. We remark that the local non-global minimizer of \eqref{eq:TRS} has the second smallest objective function value among all KKT points \cite{Wang9022}.
\begin{example}\label{ex:2}
Consider the parameterized instance of \eqref{eq:TRS}$_{\tau}$  with $n=2$ and the objective function
	\begin{equation}\label{eq:ex2}
q_{\tau}(x)=\frac{13}{2}x_1^2-\frac{13}{2}x_2^2+4x_1+\tau{\left(x_2-\frac{\sqrt{165}}{13}\right)}^2,
	\end{equation}
where $\tau\ge0$ is a parameter. When $\tau=0$, there are two global minimizers of \eqref{eq:TRS}$_{0}$,	$\bar x=(-2/13, \sqrt{165}/13)$ and $\tilde x=(-2/13, -\sqrt{165}/13)$. For any sufficiently small $\tau>0$, it holds that
\[q_{\tau}(x)\ge q_{0}(x)\ge q_{0}(\bar x)=q_{\tau}(\bar x),~ {\rm\ for\ all\ } x\in B_2,\]
and hence $\bar x$ remains globally optimal for \eqref{eq:TRS}$_{\tau}$. For such sufficiently small $\tau>0$,
we can verify that \eqref{eq:TRS}$_{\tau}$ has a local non-global minimizer, denoted by $\tilde x(\tau)$, satisfying
that  $\lim_{ \tau\rightarrow 0^+ }\tilde x(\tau)=\tilde x$. Then we have
\[
\lim_{ \tau\rightarrow 0^+ } q_{\tau}(\tilde x(\tau))-q_{0}(\bar x)=
\lim_{\tau\rightarrow 0^+ } q_{\tau}(\tilde x(\tau))-q_{0}(\tilde x) =0.
\]
The computational inaccuracy could make double-start \eqref{eq:pro} for solving \eqref{eq:TRS}$_{\tau}$ output $\tilde x(\tau)$ as the global minimizer. Note that $\tilde x(\tau)$ is far from the true global minimizer $\bar x$.
\end{example}

\section{Globally solving \eqref{eq:TRS} via \eqref{eq:pro}}\label{sec:alg}
We employ \eqref{eq:pro} to solve a novel reformulation of \eqref{eq:TRS}. The iterative sequence almost always converges to its global minimizer. The local convergence rate is linear for the ``easy case'' without assuming that we have the information that ``easy case''  holds. The global minimizer of \eqref{eq:TRS} can be quickly recovered from that of the reformulation.
%To overcome the above two disadvantages, new technique and algorithm are proposed in this section. In Section 3.1, we use the lifting technique to get a $2n$ dimension \eqref{eq:TRS} reformulation. In Section 3.2, we prove that applying
\subsection{A novel reformulation of \eqref{eq:TRS}}\label{subsec:alg}
The two-stage algorithm proposed in \cite{Ho-Nguyen17,wang17} is based on the  equivalent convex reformulation \eqref{eq:C} of \eqref{eq:TRS}, which requires calculating the minimum eigenvalue $\lambda_1$ in advance. For details, we have the following lemma.
\begin{lemma}{\rm\cite[Theorem 1]{Flippo96}}\label{le:wang17}
\eqref{eq:TRS} is equivalent to the convex programming relaxation \eqref{eq:C} in the sense that $v\eqref{eq:TRS}= v\eqref{eq:C}$. Moreover, suppose $\lambda_1<0$ and let $v_1$ be the corresponding eigenvector. Then for any optimal solution of \eqref{eq:C}, say $\tilde x$, we can
explicitly recover the global minimizer of \eqref{eq:TRS} $x^*$ as follows
	\begin{equation}\label{eq:sol}
		x^*=\tilde x+\frac{\sqrt{{(\tilde x^{T}v_1)}^2-v_1^Tv_1(\tilde x^{T}\tilde x-1)}-\tilde x^{T}v_1}{v_1^Tv_1}v_1.
	\end{equation}
\end{lemma}
Note that the cost of building \eqref{eq:C} is almost as high as solving \eqref{eq:TRS} itself, since computing $\min\{\lambda_1/2,0\}$ is equivalent to solving the following homogeneous \eqref{eq:TRS}:
\begin{equation*}
\begin{array}{cl}
\min\{\frac{\lambda_1}{2},0\}=\min_{y\in \R^n} \{\frac{1}{2}y^THy:~y^Ty\le1\}.
\end{array}
\end{equation*}
%This motivates us to search for a cheaper reformulation.
Accordingly,
for any given $x\in B_n$,  the additional term in the objective function of \eqref{eq:C} can be rewritten as a homogeneous \eqref{eq:TRS}:
\begin{equation*}
	\begin{array}{cl}
\min\{\frac{\lambda_{1}}{2},0\}(1- x^Tx)=\min_{y\in \R^n}\{\frac{1}{2}y^THy:\,y^Ty\le 1-x^{T}x\}.
	\end{array}
\end{equation*}
Therefore, \eqref{eq:C} is equivalent to the joint minimization problem in $\R^{2n}$:
\begin{equation}\label{eq:12}
	\begin{array}{cl}
		\min_{(x,y)\in\R^{2n}} \{q(x)+\frac{1}{2}y^THy:\, x\in B_n,\ y^Ty\le 1-x^{T}x\}.
	\end{array}
\end{equation}
Since the constraint $x\in B_n$ is redundant as
\[
y^Ty\le 1-x^{T}x~\Longrightarrow~ x^{T}x\le 1 ~({\rm i.e.,}~x\in B_n),
\]
the joint minimization problem \eqref{eq:12} equivalently reduces to  	
\begin{equation}\label{eq:C1}\tag{D}
	\begin{array}{cl}
		\min  \{\frac{1}{2}x^THx+\frac{1}{2}y^THy+c^Tx:\, (x,y)\in  B_{2n} \},
	\end{array}
\end{equation}
which remains a (possibly nonconvex) \eqref{eq:TRS} in $\R^{2n}$.
Though at the cost of double variables,
building \eqref{eq:C1} no longer depends on computing $\lambda_1$ when compared with \eqref{eq:C}.

\begin{theorem} \label{th:reformu}
\eqref{eq:TRS} is equivalent to \eqref{eq:C1} in the sense that $v\eqref{eq:TRS}= v\eqref{eq:C1}$.
Let $(\tilde x, \tilde y)$ be an global minimizer of \eqref{eq:C1}. Then, we can explicitly recover the global minimizer of \eqref{eq:TRS} $x^*$ as follows:
\begin{itemize}
\item[(i)] if $\tilde y=0$, then $x^*=\tilde x$,
\item[(ii)] if $\tilde y\neq 0$, then the ``hard case (ii)'' holds for \eqref{eq:TRS} and
\begin{equation}\label{eq:ysol1}
x^*=\tilde x+\frac{\sqrt{{(\tilde x^T\tilde y)}^2+\tilde y^{T}\tilde y(1-\tilde x^{T}\tilde x)}-\tilde x^{T}\tilde y}{\tilde y^{T}\tilde y}\tilde y.
\end{equation}
%	globally solves \eqref{eq:TRS}.
	\end{itemize}
\end{theorem}
\begin{proof}
According to the above derivation of \eqref{eq:C1}, it is sufficient to prove the second part on recovering the optimal solution of \eqref{eq:TRS}.
	According to Lemma \ref{le:global}, $x^*$ is a global minimizer of \eqref{eq:TRS} if and only if there exists $\lambda^*\ge-\lambda_1$ such that \eqref{eq:stationary}-\eqref{eq:com} holds.
Since \eqref{eq:C1} is a \eqref{eq:TRS} in $\R^{2n}$, for a global minimizer of \eqref{eq:C1}, denoted by $(\tilde x, \tilde y)$, there exists $\tilde\lambda\ge\max\{-\lambda_1,0\}$ such that
	\begin{eqnarray}
		&&(H+\tilde\lambda I) \tilde x+c=0,\ (H+\tilde\lambda I) \tilde y=0,\label{eq:C1y_gra}\\
		&&\tilde\lambda(\tilde x^{T}\tilde x+\tilde y^{T}\tilde y-1)=0,\ \tilde x^{T}\tilde x+\tilde y^{T}\tilde y-1\le0.\label{eq:C1_com}
	\end{eqnarray}
If $\tilde y=0$, set $x^*=\tilde x$ and $\lambda^*=\tilde \lambda$, then \eqref{eq:C1y_gra}-\eqref{eq:C1_com} reduce to \eqref{eq:stationary}-\eqref{eq:com} with $\lambda^*\ge\max\{-\lambda_1,0\}$. Hence $\tilde x$ is a global minimizer of \eqref{eq:TRS}.

If $\tilde y\neq 0$, by \eqref{eq:C1y_gra}, we have $\tilde\lambda=-\lambda_1$ and %$\tilde y$ is an eigenvector of $H$ corresponding to $-\lambda_1$.
then $(H-\lambda_1 I) \tilde x+c=0$. It follows that
	$${\|(H-\lambda_1 I)^\dagger c\|}^2\le\tilde x^T\tilde x <1,$$
where the last inequality holds since $(\tilde x,\tilde y)\in  B_{2n}$  and $\tilde y\neq 0$.
	 Therefore,  ``hard case (ii)'' defined in Table \ref{tab:tab1} holds.
With $x^*$ defined in \eqref{eq:ysol1} and $\lambda^*=\tilde\lambda$, \eqref{eq:stationary} holds by \eqref{eq:ysol1}-\eqref{eq:C1y_gra}. We can also verify that $\|x^*\|=1$ so  that \eqref{eq:com} holds.  Therefore, according to Lemma \ref{le:global}, $x^*$ defined in \eqref{eq:ysol1} globally solves \eqref{eq:TRS}.
\end{proof}

\begin{rem}
The following standard semidefinite programming (SDP) relaxation for \eqref{eq:TRS},
\[\begin{array}{cl}
			\min  & \frac{1}{2} {\rm tr}(HX)+c^Tx\\
			\ST  &  {\rm tr}(X) \le1,\\
			& X-xx^T \succeq0,
		\end{array}
	\]
is tight, see \cite{Polik07,Xia2020} and references therein. It follows that the above SDP relaxation always has a rank-one solution, that is, $X=xx^T$ holds at an optimal solution. So we can add the redundant rank constraint
\[
{\rm rank}(X-xx^T)\le 1 ~\Longleftrightarrow~X=xx^T+yy^T~({\rm as}~X-xx^T \succeq0)
\]
to the above SDP relaxation and then rebuild our new reformulation  \eqref{eq:C1}.
\end{rem}

\begin{rem}
As a direct corollary of the main result in \cite{Beck06},
\eqref{eq:TRS} is  equivalent to its complex relaxation:
\begin{equation}\label{eq:comp}
	\begin{array}{cl}
		\min  \{\frac{1}{2}z^HHz+\mathcal{R}(c^Hz):\, z^Hz\le1,~z\in \mathbb{C}^n\},
	\end{array}
\end{equation}
where  $\mathbb{C}^n$ is the $n$-dimensional complex vector space, $\mathcal{R}(\cdot)$ (resp., $\mathcal{I}(\cdot)$) denotes the real  (resp., image) part of $(\cdot)$.
By introducing  $x=\mathcal{R}(z)$ and  $y=\mathcal{I}(z)$, we recover
the new reformulation \eqref{eq:C1} from \eqref{eq:comp}.
\end{rem}

Though \eqref{eq:C1} remains nonconvex if \eqref{eq:TRS} is, it enjoys a nice property that \eqref{eq:TRS} does not have.
We call $(x^*,y^*)$ a second-order stationary point of \eqref{eq:C1} if the standard second-order necessary optimality condition \cite{Fletcher1987,Luenberger1984} holds at the stationary point  $(x^*,y^*)$.

%In the optimization theory, is necessary for any local minimizer of \eqref{eq:TRS}, see . As a special structural \eqref{eq:TRS} in $\R^{2n}$, \eqref{eq:C1} enjoys the following result.
\begin{proposition}\label{pr:ne}
Any second-order stationary point of \eqref{eq:C1} is globally optimal. It follows that   \eqref{eq:C1} has no local non-global minimizer.
\end{proposition}
\begin{proof}
Let $(x^*,y^*)$ be a second-order stationary point of \eqref{eq:C1} associated with the KKT multiplier $\lambda^*/2$. If $x^{*T}x^*+ y^{*T}y^*<1$, then by \eqref{eq:com} we have $\lambda^*=0$, and  the second-order necessary optimality condition reduces to  $H\succeq 0$. According to Lemma \ref{le:global}, $(x^*,y^*)$ is a global minimizer of \eqref{eq:C1}. Now we consider the case $x^{*T}x^*+ y^{*T}y^*=1$. According to the standard  second-order necessary optimality condition, we have
\begin{equation}
u^T(H+\lambda^* I)u+v^T(H+\lambda^* I)v\ge0,~\ \forall (u,v)\neq(0,0):\ u^Tx^*+v^Ty^*=0. \label{uv}
\end{equation}
We claim that $H+\lambda^* I\succeq 0$. Suppose, on the contrary,  $H+\lambda^* I$ has a negative eigenvalue associated with the eigenvector $w\neq0$.
If $w^Tx^*=w^Ty^*=0$ (resp. $(w^Tx^*)^2+(w^Ty^*)^2\neq0$), setting $u=v=w$ (resp. $u=-(w^Ty^*)w$ and $v=(w^Tx^*)w$) in \eqref{uv} yields a contradiction.
%If $w^Tx^*=w^Ty^*$ (or $w^Tx^*=-w^Ty^*$), setting $u=w,~v=-w$ (or $u=v=w$) in \eqref{uv} yields a contradiction. Otherwise, without loss of generality, we assume $(w^Tx^*)^2>(w^Ty^*)^2$. Then
%setting $u=-(w^Ty^*)w$ and $v=(w^Tx^*)w$ in \eqref{uv} yields a contradiction.  Therefore, it holds that $H+\lambda^* I\succeq 0$ and $(x^*,y^*)$ is a global minimizer of \eqref{eq:C1} according to Lemma \ref{le:global}.
\end{proof}

\subsection{Projected gradient method for globally solving \eqref{eq:C1}} Let
\begin{equation}
A=
	\left[\begin{array}{cc}
     H& 0\\ 0&H
	\end{array}\right], ~a=
	\left[\begin{array}{c}
     c\\ 0
	\end{array}\right] \in\R^{2n}. \label{Aa}
\end{equation}
We can rewrite \eqref{eq:C1}  as the following \eqref{eq:TRS} in $\R^{2n}$:
\begin{equation}\label{eq:z}\tag{D$'$}
	\begin{array}{cl}
		\min \{f(z)=\frac{1}{2}z^TAz+a^Tz:\,  z\in B_{2n}\}.
	\end{array}
\end{equation}
Employing \eqref{eq:pro} to solve \eqref{eq:z} gives the following iterative formula:
\begin{equation}\label{eq:pro1}\tag{PG2}
	z^{k+1}=P_{B_{2n}}(z^k-\eta\nabla f(z^k))=\left\{
	\begin{aligned}
		&z^k-\eta\nabla f(z^k),&&\ {\rm if}\   \|z^k-\eta\nabla f(z^k)\|\le1\\
		&\frac{z^k-\eta\nabla f(z^k)}{\|z^k-\eta\nabla f(z^k)\|},&&\ {\rm otherwise},
	\end{aligned}
	\right.
\end{equation}
where $\eta\in(0, 2/L)$ is the constant step size, $L=\|A\|_2=\|H\|_2$ is the same constant as that in \eqref{eq:pro}. The following convergence result on \eqref{eq:pro1} has been established in the literature.
\begin{lemma}{\rm(\cite[Theorem 10.15]{Beck2017First}, \cite[Theorem 4.5]{Beck2018})}\label{le:conv}
	Let $\{z^k\}_{k=0}^{\infty}$ be a sequence generated by \eqref{eq:pro1} with $\eta\in(0, 2/L)$. Then $\{z^k\}_{k=0}^{\infty}$ converges to a stationary point of \eqref{eq:z}.
\end{lemma}

Surprisingly, we can further prove that the sequence generated by \eqref{eq:pro1} almost always converges to a global minimizer of \eqref{eq:z}.
\begin{theorem}\label{th:p1}
For any $\eta\in(0, 2/L)$, let $\{z^k\}_{k=0}^{\infty}$ be a sequence generated by \eqref{eq:pro1} with $z^0$ being uniformly and randomly selected from $B_{2n}$.  Then $\{z^k\}_{k=0}^{\infty}$  converges to a global minimizer of \eqref{eq:z} with probability one.
\end{theorem}
\begin{proof}
According to Lemma \ref{le:conv}, for any initial point $z^0\in B_n$, $\{z^k\}_{k=0}^{\infty}$ generated by \eqref{eq:pro1} converges to a stationary point of \eqref{eq:z}.
When $\lambda_1\ge0$, \eqref{eq:z} is a convex optimization problem, any stationary point is a global minimizer. We assume $\lambda_1<0$.
%Let $(\bar z, \bar\lambda/2)$ be any KKT point, and $\bar z$ is not a global minimizer.
The following proof consists of two parts. First, we  prove that %there exists a zero measure set related to $\bar\lambda$ such that
$\{z^k\}_{k=0}^{\infty}$ converges to a non-globally-minimal stationary point $\bar z$ only if the initial point $z^0$ is included in a zero measure set related to $\bar\lambda$, where $\bar\lambda/2$ is the KKT multiplier associated with $\bar z$.
Second, we prove that \eqref{eq:z} has only a finite  number of KKT multipliers, which will complete the proof
since the measure of the union of a finite number of zero measure sets remains zero.

{\bf Proof of the first part.} Since $(\bar z, \bar\lambda/2)$ is a KKT point of \eqref{eq:z}, and $\bar z\in B_{2n}$ is not a global minimizer, according to Lemma \ref{le:global}, it holds that
	\begin{eqnarray}
		&&\bar\lambda\ge0,\label{barLambda0}\\
        && (A+\bar\lambda I)\bar z+a=0,\label{eq:10}\\
		&& A+\bar\lambda I \nsucceq 0,\label{eq:101}\\
		&& \bar\lambda(\bar z^T\bar z-1)=0.\label{eq:1011}
	\end{eqnarray}
The minimum eigenvalue of $A+\bar\lambda I$ reads as
%	\begin{equation}\label{eq:tilde}
$\tilde\lambda_1=\lambda_1+\bar\lambda$.
%	\end{equation}
It follows from \eqref{eq:101} that
\begin{equation}
\tilde\lambda_1<0. \label{eq:102}
\end{equation}
Define  two subspaces
\begin{eqnarray}
V_1&=&\{v\in\R^{2n}:~ (A+\bar\lambda I)v=\tilde\lambda_1v\},\label{eq:V1}\\
V_2&=&\{v\in\R^{2n}:~ v^T{(A+\bar\lambda I)}^\dagger a=0\},\label{eq:V2}
\end{eqnarray}
where $V_1$ is the eigenspace of $A+\bar\lambda I$ associated with the minimum eigenvalue $\tilde\lambda_1$, and $V_2$ is a hyperplane if ${(A+\bar\lambda I)}^\dagger a\neq 0$ and the whole space otherwise. Clearly, the dimension of $V_2$ is at least $2n-1$.
By the definition of $A$ \eqref{Aa}, the dimension of $V_1$ is at least two.
Since $(2n-1)+2>2n$,  the intersection of $V_1$ and $V_2$ is a nontrivial subspace,  and hence
there exists a nonzero vector $v_{\bar\lambda}$ such that $v_{\bar\lambda}\in V_1\cap V_2$.
Note that given the data in \eqref{eq:TRS}, $v_{\bar\lambda}$  depends only on $\bar\lambda$. %Although KKT multiplier $\bar\lambda$ may correspond to more than one stationary point, $V_1$ and $V_2$ do not change with different ones.
We can further verify that
\[
\begin{array}{lll}
{v_{\bar\lambda}}^T\bar z&={v_{\bar\lambda}}^T[\bar z+{(A+\bar\lambda I)}^\dagger a]&({\rm as}~v_{\bar\lambda}\in  V_2)\\
&=[(A+\bar\lambda I)v_{\bar\lambda}]^T[\bar z+{(A+\bar\lambda I)}^\dagger a]/\tilde\lambda_1&({\rm as}~v_{\bar\lambda}\in  V_1~{\rm and}~\tilde\lambda_1\neq 0~\eqref{eq:102})\\
&={v_{\bar\lambda}}^T[(A+\bar\lambda I)\bar z+a]/\tilde\lambda_1&({\rm as}~a\in {\rm Range}(A+\bar\lambda I)~{\rm by}~\eqref{eq:10})\\
&=0.&({\rm by}~\eqref{eq:10})
\end{array}
\]
%holds. The last equality holds, since either $\bar z+{(A+\bar\lambda I)}^\dagger a=0$ holds or $\bar z+{(A+\bar\lambda I)}^\dagger a$ and $v_{\bar\lambda}$ are eigenvectors of $A+\bar\lambda I$ corresponding to its zero eigenvalue by \eqref{eq:10} and its negative eigenvalue $\tilde\lambda_1$ by \eqref{eq:V1}, respectively. The penultimate equality is due to $v_{\bar\lambda}\in  V_2$.
In sum,  we obtain
\begin{eqnarray}
v_{\bar\lambda}\neq 0,\ (A+\bar\lambda I)v_{\bar\lambda}=\tilde\lambda_1v_{\bar\lambda},~v_{\bar\lambda}^T\bar z=0.
\label{eq:v1}
\end{eqnarray}
We claim that the sequence $\{z^k\}_{k=0}^{\infty}$ generated by \eqref{eq:pro1} converges to $\bar z$ only if the initial point $z^0$ satisfies $v_{\bar\lambda}^T z^0=0$.
Suppose on the contrary that
\begin{equation}
v_{\bar\lambda}^T z^0\neq0. \label{vz0}
\end{equation}
We first rewrite the iteration formulation as
\begin{equation}
z^{k+1}=\frac{z^k-\eta(Az^k+a)}{r^k}, ~k = 0, 1, 2,\cdots, \label{zk}
\end{equation}
where $r^k:=\max\{\|z^k-\eta(Az^k+a)\|, 1\}$. According to \eqref{eq:v1}, we have
	\begin{equation}\label{eq:11}
		v_{\bar\lambda}^Ta=\tilde\lambda_1v_{\bar\lambda}^T\bar z+v_{\bar\lambda}^Ta=v_{\bar\lambda}^T(A+\bar\lambda I)\bar z+v_{\bar\lambda}^Ta=v_{\bar\lambda}^T[(A+\bar\lambda I)\bar z+a]=0,
	\end{equation}
where the last equality holds due to \eqref{eq:10}.
Multiplying both sides of \eqref{zk} with $v_{\bar\lambda}^T $ from left yields that
	\begin{eqnarray}
		v_{\bar\lambda}^T z^{k+1}&=&\frac{v_{\bar\lambda}^Tz^k-\eta(v_{\bar\lambda}^TAz^k+v_{\bar\lambda}^Ta)}{r^k}\nonumber\\
		&=&\frac{v_{\bar\lambda}^Tz^k-\eta[v_{\bar\lambda}^T(A+\bar\lambda I)z^k-\bar\lambda v_{\bar\lambda}^Tz^k]}{r^k}\label{eq:sec}\\
		&=&\frac{v_{\bar\lambda}^Tz^k-\eta[\tilde\lambda_1v_{\bar\lambda}^Tz^k-\bar\lambda v_{\bar\lambda}^Tz^k]}{r^k}\nonumber\\
		&=&\frac{(1-\eta\tilde\lambda_1+\eta\bar\lambda)v_{\bar\lambda}^Tz^k}{r^k},\label{eq:13}		
	\end{eqnarray}
	where  \eqref{eq:sec}  follows from \eqref{eq:11}. We also have
\begin{eqnarray}
		{\rm lim}_{k\rightarrow\infty}r^k&={\rm lim}_{k\rightarrow\infty}\max\{\|z^k-\eta(Az^k+a)\|, 1\}\nonumber\\
		&=\max\{\|\bar z-\eta(A\bar z+a)\|, 1\}\nonumber\\
		&=\max\{(1+\eta\bar\lambda)\|\bar z\|, 1\},\label{eq:14}
\end{eqnarray}
	where \eqref{eq:14} is due to \eqref{eq:10}.
%As we have assumed   $\lambda_1<0$, it  follows from $\lambda_1<0$  that $1-\eta\lambda_1>1$.
By \eqref{barLambda0}, \eqref{eq:102} and $\eta>0$, we have $1-\eta\tilde\lambda_1+\eta\bar\lambda>1$.
According to  \eqref{eq:13}, the sign of $v_{\bar\lambda}^T z^k$ remains unchanged for all $k$. It follows from \eqref{vz0} that $v_{\bar\lambda}^T z^k\neq0$ for $k=0,1,2,\cdots$.
	By bringing \eqref{eq:14} into \eqref{eq:13}, we have
	\begin{equation}{\rm lim}_{k\rightarrow\infty}\frac{v_{\bar\lambda}^T z^{k+1}}{v_{\bar\lambda}^Tz^k}=\frac{(1-\eta\tilde\lambda_1+\eta\bar\lambda)}{\max\{(1+\eta\bar\lambda)\|\bar z\|, 1\}}>1,\label{eq:15}
	\end{equation}
	where the inequality follows from $\eta>0$, \eqref{barLambda0}, \eqref{eq:102} and $\bar z\in B_{2n}$. Therefore, under the assumption \eqref{vz0}, it holds that ${\rm lim}_{k\rightarrow\infty}v_{\bar\lambda}^T z^k\neq0$, which contradicts the last equality in \eqref{eq:v1}. We conclude that $\{z^k\}_{k=0}^{\infty}$ converges to $\bar z$ only if $v_{\bar\lambda}^T z^0=0$.
	
{\bf Proof of the second part.} We show that \eqref{eq:z} has only a finite number of KKT multipliers.
Let $A=U\Lambda U^T$ be an eigenvalue decomposition of  $A$, where $\Lambda=\diag(\lambda_1,\lambda_1,\lambda_2,\lambda_2, \cdots, \lambda_n,\lambda_n)$ with $\lambda_1\le\cdots\le\lambda_n$ and $U=[u_1,\cdots,u_{2n}]\in\R^{(2n)\times (2n)}$ is orthonormal. Suppose that $(\bar z, \bar\lambda/2)$ is a KKT point of \eqref{eq:z} and  $\bar z$ is not a global  minimizer of \eqref{eq:z}.
Then we have \eqref{barLambda0}-\eqref{eq:1011}.
%Suppose that $\bar\lambda/2$ is a KKT multiplier  corresponding to stationary point, which is not a global minimizer.
According to Lemma \ref{le:global}, there are three possible cases for $\bar\lambda$:
	\begin{itemize}
	 \item[(a)] $\bar\lambda=0$;
	 \item[(b)] $\bar\lambda=-\lambda_i$ for some $i\in\{2, \cdots, n\}$;
	 \item[(c)] $\bar\lambda\neq-\lambda_i$ for all $i\in\{1, \cdots, n\}$, and
	 \begin{equation}\label{eq:equation}
	 {\|{(A+\bar\lambda I)}^{-1}a\|}^2=1.
	 \end{equation}
%In case (c), by \eqref{eq:equation}, $\bar\lambda$ is a zero point of
Equivalently, $\bar\lambda$ is a zero point of
	 \begin{equation}\label{eq:se} \phi(\lambda)=\sum_{i=1}^n\frac{( u_{2i-1}^Ta)^2+( u_{2i}^Ta)^2}{{(\lambda_i+\lambda)}^2}-1.
	\end{equation}
\end{itemize}
It is proved in \cite{Mart94} that $\phi(\lambda)$ defined in \eqref{eq:se} is strictly convex in $(-\lambda_{i+1}, -\lambda_{i})$ for $i=1, \cdots, n-1$, and is monotone in $(-\lambda_1, +\infty)$ and $(-\infty, -\lambda_n)$, respectively. Thus,  there are at most $n-1$ and $2n-1$ KKT multipliers smaller than $-\lambda_1$ for cases (b) and (c), respectively.
In sum, \eqref{eq:z} has at most $3n-1$  KKT multipliers corresponding to non-globally-minimal stationary points.
Denote by  $M$  the set of all these KKT multipliers.
Let
\begin{equation}
Z=\bigcup_{\bar\lambda/2\in  M }\{z\in\R^{2n}:~ z^Tv_{\bar\lambda}=0\},\label{ZZ}
\end{equation}
where
$v_{\bar\lambda}$ is a nonzero point in $V_1\cap V_2$ defined in \eqref{eq:V1}-\eqref{eq:V2}.
According to the proof of the first part,
as long as $z^0\not\in Z$, $\{z^k\}_{k=0}^{\infty}$ generated by \eqref{eq:pro1} does not converge to $\bar z$ associated with any KKT multiplier $\bar\lambda/2\in M$.
That is, the convergence point must be a global minimizer.  We complete the proof by noting that the measure of the set $Z$ \eqref{ZZ} is zero.
%for any KKT point $(\bar z, \bar\lambda/2)$ where $\bar\lambda\in K$, there exists a nonzero vector $v_{\bar\lambda}$ in $V_1\cap V_2$ which is defined in \eqref{eq:V1}-\eqref{eq:V2}.
%	 As long as $z^0\notin\{z: z^Tv_{\bar\lambda}=0\}$, then $\{z^k\}_{k=0}^{\infty}$ generated by \eqref{eq:pro1} does not converge to $\bar z$. Therefore, when $z^0$ is not in the union of at most $3n-2$ subspaces, namely $z^0\notin\cup_{\bar\lambda\in K}\{z: z^Tv_{\bar\lambda}=0\}$, $\{z^k\}_{k=0}^{\infty}$ generated by \eqref{eq:pro1} converges to a global minimizer of \eqref{eq:z}.
\end{proof}

\subsection{Detailed algorithm and local convergence rate}\label{sec:li}
To globally solve \eqref{eq:TRS}, we first employ \eqref{eq:pro1} to solve \eqref{eq:z}, and then recover the global minimizer of \eqref{eq:TRS} according to  Theorem \ref{th:reformu}. For completeness, we summarize the new one-stage single-start algorithm in Algorithm \ref{al:1}.
\begin{algorithm}
	\caption{Framework of globally solving \eqref{eq:TRS}.}\label{al:1}
	\begin{algorithmic}[1]
\REQUIRE{The input data of \eqref{eq:TRS}: $n, H, c$; the step size:  $\eta \in(0,2/\|H\|_2)$.}
\ENSURE{Approximation of the global minimizer of \eqref{eq:TRS}: $x^*$.}
\STATE	Initialize $(x^0, y^0)\in B_{2n}$ and the iteration number $k=0$.
	
	\WHILE{not converged}
		\STATE $\bar x^k=x^k-\eta(Hx^k+c),\  \bar y^k=y^k-\eta(Hy^k),\ d^k=\sqrt{(\bar x^k)^T\bar x^k+(\bar y^k)^T\bar y^k}$;
		
		\IF{$d^k> 1$}
			\STATE $x^{k+1}=\frac{\bar x^k}{d^k},~  y^{k+1}=\frac{\bar y^k}{d^k}$;
			\ELSE
				\STATE $x^{k+1}=\bar x^k,~ y^{k+1}=\bar y^k$;
	
	   \ENDIF
	 \STATE $k=k+1$;

\ENDWHILE
\STATE Denote by $(\tilde x, \tilde y)$ the convergence point of $\{(x^k, y^k)\}_{k=0}^{\infty}$.
\IF{$\tilde y=0$}
\STATE Output $x^*=\tilde x$.
\ELSE
	\STATE Compute $\theta=\frac{\sqrt{{(\tilde x^T\tilde y)}^2-\tilde y^{T}\tilde y(\tilde x^{T}\tilde x-1)}-\tilde x^{T}\tilde y}{\tilde y^{T}\tilde y}$, and output  $x^*=\tilde x+\theta\tilde y$.
	\ENDIF
\end{algorithmic}
\end{algorithm}

\begin{rem}
Initialized with $y^0=0$, Algorithm \ref{al:1} reduces to \eqref{eq:pro}.
\end{rem}
\begin{rem}
Comparing with the two-stage algorithm \cite{Ho-Nguyen17,wang17} based on the equivalent convex  reformulation \eqref{eq:C}, Algorithm \ref{al:1} can be regarded as a hybrid of power method and \eqref{eq:pro}. The iteration on the artificial variable $y^k$ (after normalizing)  corresponds to that of the power method for finding the largest dominant eigenvalue of $I-\eta H$.
In fact, when $\lambda_1\le0$, the largest dominant eigenvalue of $I-\eta H$ is $1-\eta\lambda_1$, since for any eigenvalue of $H$, $\lambda_i$, it holds that $\lambda_1\le\lambda_i\le L$ and hence
\[
-1\le 1-\eta \lambda_i\le 1-\eta \lambda_1~{\rm and}~1-\eta \lambda_1\ge1
\]
for $\eta \in(0,2/L]$.
Therefore, in case of $\lambda_1\le0$,  the sequence $\{y^k/\|y^k\|\}_{k=0}^{\infty}$ converges to $v_1$, the unit-eigenvector of $H$ corresponding to $\lambda_1$.
\end{rem}
%Jiang and Li \cite[Theorem 5.1]{Jiang22} established the local convergence rate of \eqref{eq:pro} for solving \eqref{eq:TRS} (see Lemma \ref{le:rate}).

With Lemma \ref{le:rate}, we can establish the local convergence rate of \eqref{eq:pro1} for solving \eqref{eq:z}.  To this end, we need the following result.

\begin{lemma}\label{lem:2}
If the ``ill case'' of \eqref{eq:TRS} holds, so does \eqref{eq:z} and vise versa.
\end{lemma}
\begin{proof}
According to Table \ref{tab:tab1}, the ``ill case'' of \eqref{eq:TRS} holds if and only if
\begin{equation}\label{eq:illx}
\lambda_1\le0,\ c\in{\rm Range(H-\lambda_1 I)}\ {\rm\ and\ } \|{(H-\lambda_1 I)}^{\dagger}c\|=1.
\end{equation}
%\begin{equation}\label
%0\in{\rm Range(A-\lambda_1 I)}.
%\end{equation}
Note that $\lambda_1=\lambda_{\min}(H)$.
By the definition of $A$ and $a$ \eqref{Aa}, we have $\lambda_{\min}(A)=\lambda_1$, and then
$a\in{\rm Range(A-\lambda_1 I)}$ if and only if
$c\in{\rm Range(H-\lambda_1 I)}$. Moreover, we have
\[
(A-\lambda_1 I)^{\dagger}a =
\left[\begin{array}{c}
     (H-\lambda_1 I)^{\dagger}c\\ 0
	\end{array}\right],
\]
and hence $\|{(A-\lambda_1 I)}^{\dagger}a\|=\|{(H-\lambda_1 I)}^{\dagger}c\|$. In sum, we obtain
\begin{equation}\label{eq:illz}
\lambda_1\le0,\ a\in{\rm Range}(A-\lambda_{\min}(A) I), {\rm\ and\ } \|{(A-\lambda_{\min}(A) I)}^{\dagger}a\|=1.
\end{equation}
Then the ``ill case'' of \eqref{eq:z} holds. The reverse part is similarly proved.
\end{proof}

By Theorem \ref{th:p1}, Lemmas \ref{le:rate} and \ref{lem:2}, we have the following result.
\begin{corollary}\label{corM}
Initialized with a point uniformly and randomly selected from $B_{2n}$, the sequence generated by Algorithm \ref{al:1} converges to a correct point (from which one can construct the global minimizer of \eqref{eq:TRS}) with probability one.
The local convergence rate is linear for the  ``well case'' and sublinear for the ``ill case''.
\end{corollary}

\begin{rem}
In comparison with Corollary \ref{corM}, the local convergence rate of the second run of double-start \eqref{eq:pro} presented in \cite{Beck2018} could be sublinear even in the ``easy case'', see Example \ref{ex:1}.
\end{rem}

\begin{rem}
Suppose $H$ has $p$ nonzero entries. The worst-case computational costs of  \eqref{eq:pro} and our Algorithm \ref{al:1} in each iteration are $2p+5n$ and $4p+9n$, respectively.
Therefore, compared with double-start \eqref{eq:pro}, Algorithm \ref{al:1} still has a slight benefit  in view of complexity per iteration as
$4p+9n < 2\times(2p+5n).$
The reason is that the objective function in \eqref{eq:z} has no linear term with respect to the artificial variable $y$.
\end{rem}

\section{Globally solving \eqref{eq:TRS_e}}\label{sec:eq}
In this section, we first present the generalized projected gradient method  for globally solving \eqref{eq:TRS_e} and then suggest a potentially more efficient approach by reformulating \eqref{eq:TRS_e} as \eqref{eq:TRS}.

\subsection{Generalized projected gradient}
We begin with a general nonconvex optimization problem with a single constraint:
\begin{equation}\label{eq:fg}
		\min\{h(x): ~x\in Q=\{x\in \R^n:~g(x)=0\}\},
\end{equation}
where $h,g: \R^n\rightarrow \R$ are continuously differentiable, $\nabla h$ is Lipschitz continuous on $Q$ with Lipschitz constant $L$.  Assume that $Q$ is compact and $Q\cap\{x: \nabla g(x)=0\}=\emptyset$ so that LICQ holds for \eqref{eq:fg}.

The generalized projected gradient  reads as
%In the $k$-th step, we first perform the gradient descent step $\bar x^k=x^k-\eta_e\nabla f(x^k)$, then let
\begin{equation}\label{eq:x}\tag{GPG}
x^{k+1}\in {\rm argmin}\{\|x-x^k\|^2:~x\in P_Q(x^k-\eta_e\nabla h(x^k))\},
\end{equation}
where $\eta_e$ is the step size.

\begin{proposition}\label{th:eTRS}
Let $\{x^k\}_{k=0}^{\infty}$ be the sequence generated by \eqref{eq:x} with $x^0\in Q$ and constant step size $\eta_e\in(0, 1/L]$. Then we have the following results.
\begin{itemize}
\item[(i)] 	For all $k$, it holds that
	\begin{eqnarray}
&& h(x^{k+1})-h(x^{k})\le
\frac{1}{\eta_e}\left({\|x^{k+1}-s^k\|}^2-{\|x^{k}-s^k\|}^2\right)
\le0,\nonumber
\end{eqnarray}
where $s^k=x^{k}-\eta_e\nabla h(x^{k})$.
	\item[(ii)] Any accumulation point of $\{x^k\}_{k=0}^{\infty}$ is a stationary point of \eqref{eq:fg}.
\end{itemize}
\end{proposition}
\begin{proof}
(i) Since $\nabla h$ is $L$-Lipschitz continuous and $\eta_e\in(0, 1/L]$, we have
	\begin{eqnarray}
 h(x^{k+1})-h(x^{k})
&\le&(\nabla h(x^k))^T(x^{k+1}-x^{k})+\frac{L}{2}{\|x^{k+1}-x^{k}\|}^2\nonumber\\
&\le&\frac{1}{\eta_e}(x^{k}-s^k)^T(x^{k+1}-x^{k})+\frac{1}{2\eta_e}{\|x^{k+1}-x^{k}\|}^2\nonumber\\		&=&\frac{1}{2\eta_e}\left({\|x^{k+1}-s^k\|}^2-{\|x^{k}-s^k\|}^2\right)\nonumber\\
&\le&0,\nonumber
\end{eqnarray}
where the last inequality follows from the selection of $x^{k+1}$ in \eqref{eq:x}.
	
(ii)   Since $Q$ is compact,  $\{x^k\}_{k=0}^{\infty}\subseteq Q$ has an accumulation point. Suppose that the subsequence $\{x^{m_k}\}$ converges to $\bar x$.
Then $\{s^{m_k}\}$ converges to $\bar s:=\bar x-\eta_e\nabla h(\bar x)$ and $h(x^{m_{k+1}})-h(x^{m_{k}})\rightarrow 0$. By (i), $h(x^{k})$ is non-increasing with $k$. Thus, $h(x^{k+1})-h(x^{k})\rightarrow 0$ as $k\rightarrow \infty$. Then it follows from (i) that
	\begin{equation}\label{eq:1}
		{\|x^{k+1}-s^k\|}^2-{\|x^{k}-s^k\|}^2\rightarrow 0,~{\rm as}~k\rightarrow \infty.
	\end{equation}
 Since $Q$ is compact, it holds that
{	\begin{equation}\label{eq:2}
	{\|x^{k+1}-s^k\|}^2={\|P_Q(s^k) -s^k\|}^2\rightarrow{\|P_Q(\bar s) -\bar s\|}^2,~{\rm as}~k\rightarrow \infty.
	\end{equation}}
	According to ${\rm lim}_{k\rightarrow+\infty}x^{m_k}=\bar x,\ {\rm lim}_{k\rightarrow+\infty}s^{m_k}=\bar s$ and \eqref{eq:1}-\eqref{eq:2}, we have
\[
{\|\bar x -\bar s\|}^2={\|P_Q(\bar s) -\bar s\|}^2,
\]
and then it holds that $\bar x\in P_Q(\bar s)=P_Q(\bar x-\eta_e \nabla h(\bar x))$. That is,
\begin{equation}\label{eq:barx}
\bar x\in {\rm argmin} \{\|x-(\bar x-\eta_e \nabla h(\bar x))\|^2:~g(x)=0\}.
\end{equation}
By  the KKT condition of \eqref{eq:barx}, there exists $\lambda\in\R$ such that $\nabla h(\bar x)+\lambda/\eta_e \nabla g(\bar x)=0$. Consequently, $\bar x$ is a stationary point of \eqref{eq:fg}.
\end{proof}

\subsection{Solving \eqref{eq:TRS_e} by \eqref{eq:x}}
Before solving \eqref{eq:TRS_e}, %\eqref{eq:x} reduces to \eqref{eq:proe}. Motivated by \eqref{eq:pro1}, we employ \eqref{eq:proe} to solve
we first present the following lifted reformulation of \eqref{eq:TRS_e} similar to \eqref{eq:TRS_e}:
\begin{equation}\label{eq:ze}\tag{De}
\min \{f(z)=\frac{1}{2}z^TAz+a^Tz:\, z\in \partial B_{2n}\},
\end{equation}
where $A$ and $a$ are defined in \eqref{Aa}. Similar to Proposition \ref{pr:ne}, \eqref{eq:ze} has the nice property that any  second-order stationary point is
a global minimizer.
\begin{rem}
Our reformulation \eqref{eq:ze} is new.  Recently,
Boumal et al. \cite{Boumal2019} proposed the following rank-two SDP reformulation for \eqref{eq:TRS_e} based on Burer-Monteiro decomposition \cite{Burer03}:
\begin{equation}\label{eq:X}
	\min_{Y_1\in\R^{n\times 2}, y_2\in\R^2} \{ {\rm tr}(CY Y^T):\, {\rm tr}( Y_1Y_1^T)=1, \, {\|y_2\|}^2=1, \,
Y^T=[Y_1^T ~ y_2]
%\left[\begin{array}{c}
%	Y_1\\ y_2^T
%\end{array}\right]
\}.
\end{equation}
It can be regarded as a heavy homogenized version of
our reformulation \eqref{eq:ze} with more variables and constraints.
%Note that the rank-two SDP relaxation approach dates back to the application for Max-Cut problem \cite{Burer02}.
\end{rem}

Motivated by \eqref{eq:pro1},
we employ \eqref{eq:proe} to solve   \eqref{eq:TRS_e} and obtain the iterative formlula
\begin{equation}\label{eq:proe2}\tag{PGe2}
	z^{k+1}=
	\frac{z^k-\eta_e\nabla f(z^k)}{\|z^k-\eta_e\nabla f(z^k)\|}.
\end{equation}
Note that \eqref{eq:proe2} is well-defined, if it holds that $z^k-\eta_e \nabla f(z^k)\neq 0$ for all $k$.
Suppose that there is a positive integer $K$ such that $z^K-\eta_e \nabla f(z^K)=0$,
we can stop the iteration as $(z^K,-1/(2\eta_e))$ is already a KKT point of \eqref{eq:ze}.
As an extension of Lemma \ref{le:conv}, we can establish the convergence result.
\begin{lemma}\label{le:decrease}
Let the sequence $\{z^{k}\}_{k\ge0}$ be generated by \eqref{eq:proe2} for solving \eqref{eq:ze} with $z^0\in B_{2n}$ and constant step size $\eta_e\in(0, 1/L]$. Then either there is a positive integer $K$ such that $z^{K}$  is a stationary point or $\{z^{k}\}_{k=0}^{\infty}$  converges to a stationary point of \eqref{eq:ze}.
\end{lemma}
\begin{proof}
Define $s^k=z^{k}-\eta_e\nabla f(z^{k})$ for $k=0, \dots, +\infty$.
As shown above,
if there is a positive integer $K$ such that $s^K=0$, then $z^K$ is already a stationary point of \eqref{eq:ze}.
Now we assume $s^k\neq 0$ for all $k\ge 0$.

Suppose that $\{z^k\}_{k=0}^{\infty}$ has two different accumulation points, say $\bar z$ and $\tilde z$. By Proposition \ref{th:eTRS} (ii), both $\bar z$ and $\tilde z$ are stationary points. Let $\bar \mu/2$ and $\tilde \mu/2$ be the KKT multipliers corresponding to $\bar z$ and $\tilde z$, respectively. According to Proposition \ref{th:eTRS} (i), $\{f(z^k)\}_{k=0}^{\infty}$ is non-increasing. Then it holds that $f(\bar z)=f(\tilde z)$. As $\bar z\neq \tilde z$, according to \cite[Theorem 1]{Gander81} or \cite[Remark 3.1]{Wang9022}, we have $\bar \mu=\tilde \mu$. For simplicity, we use $\mu$ to represent $\bar \mu$ and $\tilde \mu$. We first write down the KKT condition:
\begin{equation}\label{eq:gra}
(A+\mu I)\bar z+a=(A+\mu I)\tilde z+a=0,\ \bar z^T\bar z=\tilde z^T\tilde z=1.
\end{equation}
Denote $\bar z=(\bar x, \bar y)$ and $\tilde z=(\tilde x, \tilde y)$.
Based on the eigenvalue decomposition $A=U\Lambda U^T$, \eqref{eq:gra} is equivalent to
\begin{equation}
(\Lambda+\mu I)U^T\bar z+U^Ta=(\Lambda+\mu I)U^T\tilde z+U^Ta=0.\label{kkt:d}
\end{equation}
Let  $I_0=\{i: ~\Lambda_{ii}+\mu=0\}$. By \eqref{kkt:d}, it holds that
\begin{equation}\label{Ux:1}
(U^T\bar z)_i=(U^T\tilde z)_i= -(U^Ta)_i/(\Lambda_{ii}+\mu ){\ \rm for\ all\ } i\not\in I_0.
\end{equation}
We obtain that $\bar z=\tilde z$ if $I_0=\emptyset$.
 Below we consider $I_0\neq\emptyset$. It follows from \eqref{kkt:d} that
\[
(U^Ta)_i=0{\ \rm for\ all\ } i\in I_0.
\]
Then the iterative scheme of $z^k$ reduces to
\[
(U^T z^{k+1})_i=\frac{(U^T z^k)_i-\eta_e\Lambda_{ii}{(U^T z^k)}_i}{\| s^k \|}=\frac{(1+\eta_e\mu){(U^T z^k)}_i}{\| s^k \|} {\ \rm for\ all\ } i\in I_0,
\]
which implies that
\begin{equation}\label{eq:i}
		(U^T z^{k})_i=\frac{{(1+\eta_e\mu)}^k}{\prod\limits_{j=0}^{k}\| s^j \|}{(U^T z^0)}_i{\ \rm for\ all\ } i\in I_0 {\rm\ and\ } k=0, \cdots, +\infty.
	\end{equation}
Since $\Lambda_{ii}+\mu=0$, $\mu\ge -L$ holds. Then, for any $\eta_e\in(0, 1/L]$, we have
\[
1+\eta_e\mu\ge 0.
\]
Therefore, there are two scalars $\alpha, \beta\ge0$ such that
\begin{equation} \label{Ux:3}
(U^T\bar z)_{i}=\alpha (U^T z^0)_i,~(U^T\tilde z)_{i}=\beta (U^T z^0)_i {\ \rm for\ all\ } i\in I_0.
\end{equation}
It follows from the facts $\|U^T\bar z\|=\|\bar z\|=1=\|\tilde z\|=\|U^T\tilde z\|$ and \eqref{Ux:1} that
\begin{equation} \label{Ux:4}
\sum_{i\in I_0} (U^T\bar z)_{i}^2=\sum_{i\in I_0} (U^T\tilde z)_{i}^2.
\end{equation}
Substituting \eqref{Ux:3} into \eqref{Ux:4} yields that  either
$\sum_{i\in I_0} (U^T z^0)_i^2=0$ or $\alpha^2=\beta^2$ holds.  In the former case,  $(U^T z^0)_i=0$  and hence $(U^T\bar z)_{i}=(U^T\tilde z)_{i}=0$ holds for all $i\in I_0$ by \eqref{eq:i}. In the latter case, we have $\alpha =\beta $ since both are nonnegative. In sum, we always have
\begin{equation} \label{Ux:2}
(U^T\bar z)_{i}=(U^T\tilde z)_{i}~{\ \rm for\ all\ } i\in I_0.
\end{equation}
Combining \eqref{Ux:1} with \eqref{Ux:2} yields the contradiction $\bar z=\tilde z$.
\end{proof}

Similar to Theorem \ref{th:p1}, we establish the following result.
\begin{theorem}\label{th:conv_e}
Initialized with a point {uniformly and randomly generated} from $\partial B_{2n}$, the sequence generated by \eqref{eq:proe2} with step size $\eta_e\in(0, 1/L]$ converges to the global minimizer of \eqref{eq:ze} with probability one.
\end{theorem}

\begin{proof}
	If $\eta_e=1/L$ and $A=L\cdot I$, then by \eqref{eq:proe2},
$$
z^1=\frac{z^0-\frac{1}{L}(Lz^0+a)}{\|z^0-\frac{1}{L}(Lz^0+a)\|}=-\frac{a}{\|a\|}.
$$
By Cauchy-Schwartz inequality, we have
\begin{equation}\label{eq:scalar}
v\eqref{eq:TRS_e}=\min \left\{\frac{L}{2}+ a^Tx:\, x\in \partial B_n \right\}=\frac{L}{2}-\|a\|_2=f(z^1).
\end{equation}
That is, the global minimizer is obtained after one step. In the following, it is  sufficient to consider the case that either $\eta_e<1/L$ or $A\neq L\cdot I$.
	
Setting the same $v_{\bar\lambda}$ as that in the proof of Theorem \ref{th:p1}, we obtain a relation similar to \eqref{eq:13} with $\eta$ being replaced by $\eta_e$.
Moreover, we have
\begin{equation}
1-\eta_e\tilde\lambda_1+\eta_e\bar\lambda=1-\eta_e\lambda_1\ge 1-\eta_e L
\ge 0. \label{etabd}
\end{equation}
In case of $A\neq L\cdot I$, it holds that $\lambda_1<L$. So we have either
$\lambda_1<L$ or $\eta_e<1/L$. It implies that at least one of the two inequalities in
\eqref{etabd} holds strictly.
Therefore, according to \eqref{eq:13}, $v_{\bar\lambda}^T z^k\neq0$ and the sign of $v_{\bar\lambda}^T z^k$ remains unchanged for all $k>0$. Moreover, \eqref{eq:15} holds true if $\eta$ is replaced by $\eta_e$. The remaining part of the proof is the same as that of Theorem \ref{th:p1}.
\end{proof}
\begin{rem}
None of the proofs of Proposition \ref{th:eTRS}, Lemma \ref{le:decrease} and Theorem \ref{th:conv_e} holds true if we  set $\eta_e> 1/L$.
\end{rem}
%\begin{corollary}
%Neither \eqref{eq:z} nor \eqref{eq:ze} has a local non-global minimizer.
%\end{corollary}
%\begin{proof}
%Since any local minimizer of \eqref{eq:z} is locally minimal for \eqref{eq:ze}, it is sufficient to study \eqref{eq:ze}. Suppose, on the contrary, \eqref{eq:ze} has a local non-global minimizer, denoted by $\widehat{x}$. Then, there is a $\epsilon>0$
%\end{proof}

\begin{rem}
	In \cite[Example 1, Theorem 2]{Polyak19}, \eqref{eq:proe} is proposed for solving unit-sphere constrained optimization problem, where  the Le\u{z}anski-Polyak-Lojasiewicz condition on the sphere is assumed to guarantee the convergence to the global minimizer. However, this assumption is too restrictive for \eqref{eq:TRS_e}.
\end{rem}

\begin{rem}
Jain and Kar claimed in  \cite[Theorem 3.3]{Jain17} that the sequence generated by \eqref{eq:x} converges to the global minimizer, under the conditions of $\alpha$-restricted strong smoothness and $\beta$-restricted strong convexity 	
\begin{equation*}
	\frac{\alpha}{2}{\|x-y\|}^2
	\le q(x)-q(y)-\langle\nabla q(y), x-y\rangle
	\le \frac{\beta}{2}{\|x-y\|}^2
\end{equation*}
with $\beta/\alpha<2$ for the objective function  $q(\cdot)$ with the constant step size $\eta_e=1/\beta$.
We point out that this claim is incorrect, as the convergence point may be a local non-global minimizer. The following is a counterexample.
\end{rem}

\begin{example}\label{ex:gPG_error}
	Consider the instance of  \eqref{eq:TRS_e} in $\R^2$:
	\begin{equation}\label{eq:ex3}
		\begin{array}{cl}
			\min  \{\frac{27}{2}x_1^2+\frac{53}{2}x_2^2-4x_1+9x_2:\, x\in \partial B_2\}.
		\end{array}
	\end{equation}
	One can verify that the $\alpha$-restricted strong smoothness and $\beta$-restricted strong convexity with $\beta/\alpha<2$ holds with $\alpha=27$ and $\beta=53$. According to Example 1.1 of \cite{Wang9020}, Problem \eqref{eq:ex3} has both global and local non-global minimizers. Either could be a convergence point of \eqref{eq:x}.
\end{example}

\begin{rem}
Lee et al. \cite{Lee2019} proved that  (RG) almost always avoid strict saddle point of \eqref{eq:TRS_e} by regarding it  as an optimization problem on the manifold. The  related step size is corrected in \cite{Zheng22}. Since \eqref{eq:TRS_e} could have
%a non-strict saddle point (see Example \ref{ex:1}) or even
a local non-global minimizer (see Example \ref{ex:gPG_error}), (RG) cannot solve \eqref{eq:TRS_e} globally with probability $1$.
\end{rem}

\subsection{\eqref{eq:TRS}-reformulation of \eqref{eq:TRS_e}}

According to \eqref{eq:scalar}, \eqref{eq:TRS_e} is trivial to solve if $H$ is a scalar matrix. Throughout this subsection, we only consider nontrivial \eqref{eq:TRS_e} where $H$ is not  a scalar matrix.  We start with a technical lemma.
\begin{lemma}\label{tau}
Let $\tau={\rm tr}(H)/n$. If $H$ is not  a scalar matrix, then $H- \tau I\not\succeq 0$ and
\begin{equation}
\|H- \tau I\|_2 <2 \|H \|_2.
\end{equation}
\end{lemma}
\begin{proof}
Let $\lambda_1\le\cdots\le\lambda_n$ be eigenvalues of $H$. As $H$ is not  a scalar matrix, it holds that $\lambda_1<\lambda_n$. Then, we have
\[
\tau=\frac{1}{n}{\rm tr}(H)=\frac{1}{n}\sum_{i=1}^n\lambda_i>  \lambda_1.
\]
It follows that $H- \tau I\not\succeq 0$. Moreover, one can verify that
\[
\|H- \tau I\|_2\le \|H \|_2 +\tau = \|H \|_2 +\frac{1}{n}\sum_{i=1}^n\lambda_i
< \|H \|_2 +\lambda_n\le 2\|H \|_2.
\]
\end{proof}

\begin{theorem} \label{thm:tau}
\eqref{eq:TRS_e} has the same global minimizer as the following \eqref{eq:TRS}:
\begin{equation}\label{eq:etrs2}
\min \left\{\frac{1}{2}x^T(H-\tau I)x+c^Tx:\, x \in B_n\right\},
\end{equation}
where $\tau={\rm tr}(H)/n$.
\end{theorem}
\begin{proof}
According to Lemma \ref{tau},  $H-\tau I\not\succeq 0$. Then any local minimizer of \eqref{eq:etrs2}, denoted by $x^*$, cannot be an interior point of $B_n$, i.e., $(x^*)^Tx^*=1$ must hold. It completes the proof of the equivalence.
\end{proof}

Based on Theorem \ref{thm:tau}, \eqref{eq:TRS_e} can be solved by directly employing Algorithm \ref{al:1} to \eqref{eq:etrs2}. Note that the supremum of the  step size is $\hat\eta=2/\hat L$, where $\hat L=\|H-\tau I\|_2$.
According to Lemma \ref{tau},   we have
\[
\eta_e\le \frac{1}{\|H\|_2} <   \frac{2}{\hat L}=\hat\eta.
\]
Thus, the benefit of this \eqref{eq:TRS}-reformulation approach is that the step size could be larger than that of \eqref{eq:proe}.
In particular, when $H$ is positive semidefinite (for example, the least square problem over a sphere \cite{Gander81}), we have
 \[
	\eta_e\le \frac{1}{\|H\|_2} < \frac{1}{\hat L}=\frac{\hat\eta}{2},
\]
that is, the step size of Algorithm \ref{al:1} could be twice that of \eqref{eq:proe}.
%We conjecture about the convergence rate of Algorithm \ref{al:2} as follows.
%\begin{conjecture}
%If the ill-case holds, Algorithm \ref{al:2} converges to the global minimizer of \eqref{eq:TRS_e} with a sublinear rate. Otherwise, it can be further improved to a linear rate.
%\end{conjecture}

\section{Conclusion}
We show that the sequence generated by the classical projected gradient method for solving a cheap but novel reformulation of the trust region subproblem (TRS) almost always converges to its global minimizer. The local convergence rate is at least linear for the ``easy case''. As an extension, we establish and analyze the generalized projected gradient method for  globally solving the similarly lifted equality-constrained (TRS), denoted by (TRSe). However, an alternative approach by globally solving a new nonconvex (TRS)-reformulation of (TRSe) via projected gradient method seems to be better in the sense that it allows a larger step size. Some problems remain open, including the local convergence rate of the generalized projected gradient method and the acceleration version of the projected gradient algorithm for solving our new reformulations of (TRS) and (TRSe).

%Finally, our approach can be employed to solve more hidden convex optimization problems such as the $p$-regularized subproblem.

\bibliographystyle{siamplain}
\bibliography{reference}
\end{document}

%% file: ex_shared.tex
\usepackage{lipsum}
\usepackage{amsfonts}
\usepackage{graphicx}
\usepackage{epstopdf}
\usepackage{algorithmic}

\usepackage{appendix}

\ifpdf
  \DeclareGraphicsExtensions{.eps,.pdf,.png,.jpg}
\else
  \DeclareGraphicsExtensions{.eps}
\fi

\newsiamremark{remark}{Remark}
\newsiamremark{hypothesis}{Hypothesis}
\crefname{hypothesis}{Hypothesis}{Hypotheses}
\newsiamthm{claim}{Claim}

\headers{On globally solving trust region subproblem via PG}{}

\title{On globally solving nonconvex trust region subproblem via projected gradient method\thanks{Submitted to the editors DATE.
\funding{This research was supported by the National Natural Science Foundation of China (Grant Nos.  12171021 and 11822103), and the Beijing Natural Science Foundation (Grant No. Z180005).}}}

\author{Mengmeng Song\thanks{School of Mathematical Sciences, Beihang University, Beijing 100191, People's Republic of China,
\{songmengmeng, yxia, sanjin\}@buaa.edu.cn.}
\and Yong Xia\footnotemark[2]
\thanks{Corresponding author. }
\and Jinyang Zheng\footnotemark[2]
}

\usepackage{amsopn}
\DeclareMathOperator{\diag}{diag}

%% file: TRS_Solving_0711.bbl
\begin{thebibliography}{10}

\bibitem{Absil08}
{\sc P.-A. Absil, R.~Mahony, and R.~Sepulchre}, {\em Optimization Algorithms on
  Matrix Manifolds}, Princeton University Press, 2008.

\bibitem{Adachi17}
{\sc S.~Adachi, S.~Iwata, Y.~Nakatsukasa, and A.~Takeda}, {\em Solving the
  trust-region subproblem by a generalized eigenvalue problem}, SIAM Journal on
  Optimization, 27 (2017), pp.~269--291.

\bibitem{Polyak19}
{\sc M.~V. Balashov, B.~T. Polyak, and A.~A. Tremba}, {\em Gradient projection
  and conditional gradient methods for constrained nonconvex minimization},
  Numerical Functional Analysis and Optimization, 41 (2020), pp.~822--849.

\bibitem{Beck2017First}
{\sc A.~Beck}, {\em First-Order Methods in Optimization}, MOS-SIAM Ser. Optim.
  25, SIAM, Philadelphia, PA, 2017.

\bibitem{Beck06}
{\sc A.~Beck and Y.~C. Eldar}, {\em Strong duality in nonconvex quadratic
  optimization with two quadratic constraints}, SIAM Journal on Optimization,
  17 (2006), pp.~844--860.

\bibitem{Beck2018}
{\sc A.~Beck and Y.~Vaisbourd}, {\em Globally solving the trust region
  subproblem using simple first-order methods}, SIAM Journal on Optimization,
  28 (2018), pp.~1951--1967.

\bibitem{Boumal2019}
{\sc N.~Boumal, V.~Voroninski, and A.~S. Bandeira}, {\em Deterministic
  guarantees for {B}urer-{M}onteiro factorizations of smooth semidefinite
  programs}, Communications on Pure and Applied Mathematics, 73 (2019),
  pp.~581--608.

\bibitem{Burer03}
{\sc S.~Burer and R.~Monteiro}, {\em A nonlinear programming algorithm for
  solving semidefinite programs via low-rank factorization}, Mathematical
  Programming, 95 (2003), pp.~329--357.

\bibitem{Carmon18}
{\sc Y.~Carmon and J.~C. Duchi}, {\em Analysis of {Krylov} subspace solutions
  of regularized non-convex quadratic problems}, in Advances in Neural
  Information Processing Systems, S.~Bengio, H.~Wallach, H.~Larochelle,
  K.~Grauman, N.~Cesa-Bianchi, and R.~Garnett, eds., vol.~31, Curran
  Associates, Inc., 2018.

\bibitem{Conn00}
{\sc A.~R. Conn, N.~I.~M. Gould, and P.~L. Toint}, {\em Trust Region Methods},
  Society for Industrial and Applied Mathematics, Philadelphia, 2000.

\bibitem{Fletcher1987}
{\sc R.~Fletcher}, {\em Practical Methods of Optimization}, John Wiley, New
  York, second~ed., 1987.

\bibitem{Flippo96}
{\sc O.~E. Flippo and B.~Jansen}, {\em Duality and sensitivity in nonconvex
  quadratic optimization over an ellipsoid}, European Journal of Operational
  Research, 94 (1996), pp.~167--178.

\bibitem{Gander81}
{\sc W.~Gander}, {\em Least squares with a quadratic constraint}, Numerische
  Mathematik, 36 (1980), pp.~291--307.

\bibitem{Gay81}
{\sc D.~M. Gay}, {\em Computing optimal locally constrained steps}, SIAM
  Journal on Scientific and Statistical Computing, 2 (1981), pp.~186--197.

\bibitem{Gould99}
{\sc N.~I.~M. Gould, S.~Lucidi, M.~Roma, and P.~L. Toint}, {\em Solving the
  trust-region subproblem using the {Lanczos} method}, SIAM Journal on
  Optimization, 9 (1999), pp.~504--525.

\bibitem{Hager01}
{\sc W.~W. Hager}, {\em Minimizing a quadratic over a sphere}, SIAM Journal on
  Optimization, 12 (2001), pp.~188--208.

\bibitem{Hazan16}
{\sc E.~Hazan and T.~Koren}, {\em A linear-time algorithm for trust region
  problems}, Mathematical Programming, 158 (2016), pp.~363--381.

\bibitem{Ho-Nguyen17}
{\sc N.~Ho-Nguyen and F.~K{\i}l{\i}n{\c c}-Karzan}, {\em A second-order cone
  based approach for solving the trust-region subproblem and its variants},
  SIAM Journal on Optimization, 27 (2017), pp.~1485--1512.

\bibitem{Jain17}
{\sc P.~Jain and P.~Kar}, {\em Non-convex optimization for machine learning},
  10 (2017), pp.~142--363.

\bibitem{Jiang22}
{\sc R.~Jiang and X.~Li}, {\em H{\"o}lderian error bounds and
  {Kurdyka-Lojasiewicz} inequality for the trust region subproblem},
  Mathematics of Operations Research,  (2022).
\newblock Published online, https://doi.org/10.1287/moor.2021.1243.

\bibitem{Lee2019}
{\sc J.~D. Lee, I.~Panageas, G.~Piliouras, M.~Simchowitz, M.~I. Jordan, and
  B.~Recht}, {\em First-order methods almost always avoid strict saddle
  points}, Mathematical Programming, 176 (2019), pp.~311--337.

\bibitem{Luenberger1984}
{\sc D.~G. Luenberger and Y.~Ye}, {\em Linear and Nonlinear Programming},
  Springer, New York, third~ed., 2008.

\bibitem{Mart94}
{\sc J.~M. Mart\'{i}nez}, {\em Local minimizers of quadratic functions on
  {E}uclidean balls and spheres}, SIAM Journal on Optimization, 4 (1994),
  pp.~159--176.

\bibitem{More83}
{\sc J.~J. Mor\'{e} and D.~C. Sorensen}, {\em Computing a trust region step},
  SIAM Journal on Scientific and Statistical Computing, 4 (1983), pp.~553--572.

\bibitem{PY2020}
{\sc A.~H. Phan, M.~Yamagishi, D.~Mandic, and A.~Cichocki}, {\em Quadratic
  programming over ellipsoids with applications to constrained linear
  regression and tensor decomposition.}, Neural Computing \& Applications,
  (2020), pp.~7097--7120.

\bibitem{Polik07}
{\sc I.~P\'{o}lik and T.~Terlaky}, {\em A survey of the {S}-lemma}, SIAM
  Review, 49 (2007), pp.~371--418.

\bibitem{Sorensen82}
{\sc D.~C. {Sorensen}}, {\em Newton's method with a model trust region
  modification}, SIAM Journal on Numerical Analysis, 19 (1982), pp.~409--426.

\bibitem{Sorensen97}
{\sc D.~C. Sorensen}, {\em Minimization of a large-scale quadratic function
  subject to a spherical constraint}, SIAM Journal on Optimization, 7 (1997),
  pp.~141--161.

\bibitem{Tao98}
{\sc P.~D. Tao and L.~T.~H. An}, {\em A {D.C.} optimization algorithm for
  solving the trust-region subproblem}, SIAM Journal on Optimization, 8 (1998),
  pp.~476--505.

\bibitem{Wang9022}
{\sc J.~Wang, M.~Song, and Y.~Xia}, {\em Trust-region and $p$-regularized
  subproblems: local nonglobal minimum is the second smallest objective
  function value among all first-order stationary points}.
\newblock arXiv:2108.07963, 2022.

\bibitem{wang17}
{\sc J.~Wang and Y.~Xia}, {\em A linear-time algorithm for the trust region
  subproblem based on hidden convexity}, Optimization Letters, 11 (2017),
  pp.~1639--1646.

\bibitem{Wang9020}
{\sc J.~Wang and Y.~Xia}, {\em Closing the gap between necessary and sufficient
  conditions for local nonglobal minimizer of trust region subproblem}, SIAM
  Journal on Optimization, 30 (2020), pp.~1980--1995.

\bibitem{yuan15}
{\sc Y.~x.~{Yuan}}, {\em Recent advances in trust region algorithms},
  Mathematical Programming, 151 (2015), pp.~249--281.

\bibitem{Xia2020}
{\sc Y.~Xia}, {\em A survey of hidden convex optimization.}, Journal of the
  Operations Research Society of China,  (2020), pp.~1--28.

\bibitem{Y92}
{\sc Y.~Ye}, {\em A new complexity result on minimization of a quadratic
  function with a sphere constraint}, in Recent Advances in Global
  Optimization, Princeton University Press, USA, 1992, pp.~19--31.

\bibitem{Zhang17}
{\sc L.-H. Zhang, C.~Shen, and R.-C. Li}, {\em On the generalized {Lanczos}
  trust-region method}, SIAM Journal on Optimization, 27 (2017),
  pp.~2110--2142.

\bibitem{Zheng22}
{\sc J.~Zheng and Y.~Xia}, {\em Comment on ``first-order methods almost always
  avoid strict saddle points''}.
\newblock arXiv:2204.00521, 2022.

\end{thebibliography}
